\documentclass[a4paper,10pt,reqno, english]{amsart}

\usepackage{amsmath,amssymb,amscd,amsthm,amsfonts}
\usepackage{graphicx,subcaption}
\usepackage{hyperref}
\usepackage{dsfont}
\usepackage{xcolor}
\usepackage{color,soul}
\usepackage[utf8]{inputenc}
\usepackage[nobysame, alphabetic]{amsrefs}
\usepackage[capitalize]{cleveref}
\usepackage{float}
\usepackage{enumerate}

\newtheorem{theorem}{Theorem}[section]
\newtheorem{corollary}[theorem]{Corollary}
\newtheorem{lemma}[theorem]{Lemma}

\newtheorem{conjecture}[theorem]{Conjecture}

\theoremstyle{definition}

\newtheorem{definition}[theorem]{Definition}

\newcommand{\rr}{\mathds{R}}

\DeclareMathOperator{\EMP}{EMP}

\title{High-dimensional envy-free partitions}



\author[Sober\'on]{Pablo Sober\'on}\address{Baruch College, City University of New York, New York, NY 10010 \newline The Graduate Center, City University of New York, NY 10016} 
\email{psoberon@gc.cuny.edu}

\author[Yu]{Christina Yu}\address{Massachusetts Institute of Technology, Cambridge, MA 02139} 
\email{chryu@mit.edu}

\thanks{Sober\'on's research is supported by NSF grant DMS 2054419 and a PSC-CUNY TRADB52 award.  Yu's research on this project was supported with grant NSF DMS 2051026.}

\begin{document}

\begin{abstract}
A vast array of envy-free results have been found for the subdivision of one-dimensional resources, such as the interval $[0,1]$.   The goal is to divide the space into $n$ pieces and distribute them among $n$ observers such that each receives their favorite pieces.  We study high-dimensional versions of these results.  We prove that several spaces of convex partitions of $\rr^d$ allow for envy-free division among any $n$ observers.  We also prove the existence of convex partitions of $\rr^d$ which allow for envy-free divisions among several groups of $n$ observers simultaneously. 
\end{abstract}

\maketitle

\section{Introduction}

The problem of fair division is central to both mathematical economics and topological combinatorics.  Given a resource to be shared between $n$ people, can we partition and distribute it so that each person receives a ``fair" share according to their subjective measure of value? The methods to solve such problems often involve topological tools \cites{RoldanPensado2022, Zivaljevic2017}.  In this manuscript, we are interested in studying how two large families of fair partition problems --- mass partitions and envy-free partitions --- relate to each other.  

In a mass partition problem, we are given a family of absolutely continuous probability measures on $\rr^d$, and a family $\mathcal{H}$ of partitions of $\rr^d$.  We will say that a measure on $\rr^d$ is \textit{absolutely continuous} if it is absolutely continuous with respect to the Lebesgue measure --- this will be the main condition on our measures.  We want to know if there exists a partition in $\mathcal{H}$ that splits each measure in a predetermined way (oftentimes, we want each measure to be split into parts of equal value).  The quintessential result is the ham sandwich theorem, attributed to Steinhaus.

\begin{theorem}[Ham Sandwich Theorem\cite{Steinhaus1938}]\label{thm:ham-sandwich}
    Given $d$ absolutely continuous probability measures $\mu_1, \dots, \mu_d$ on $\rr^d$, there exists a hyperplane $H$ such that its two closed halfspaces $H^+$, $H^-$ satisfy
    \[
    \mu_i (H^+) = \mu_i (H^-) = \frac{1}{2} \qquad \mbox{for }i=1,\dots,d.
    \]
\end{theorem}

A measure can be interpreted as the way in which one person assigns value to different sets.  This means that for any $d$ observers, there is a way to split $\rr^d$ into two parts by a hyperplane such that each person agrees that the two parts have equal value.  We call such a partition that splits each measure into parts of equal value an \textit{equipartition}. 

In the problem of envy-free division, we want to divide a resource among $k$ observers such that each person thinks they got the best piece (though they may not agree that everyone received the same value). In other words, no person is envious of any other person; every person thinks their share is at least as good as anyone else's. The central envy-free result is the cake-splitting theorem.  We denote by $[n]$ the set $\{1,2,\dots, n\}$.

\begin{theorem}[Cake-splitting theorem; Stromquist 1980, Woodall 1980\cites{Stromquist1980, Woodall1980}]\label{thm:cake-partition}
    Let $\mu_1, \dots, \mu_n$ be absolutely continuous probablity measures in $[0,1]$.  Then, there exists a partition of $[0,1]$ into $n$ intervals $I_1, \dots, I_n$ and a permutation $\pi:[n] \to [n]$ such that
    \[
    \mu_i (I_{\pi(i)}) \ge \mu_i (I_{\pi(i')}) \qquad \mbox{for all }i,i'  \in [n].
    \]
\end{theorem}

Here, the permutation $\pi$ designates how the pieces should be allocated between the $n$ observers. Then, each person $i$ prefers the piece $\pi(i)$ that they receive, so this division is envy-free. More generally, we say a partition $(C_1,\dots, C_n)$ of $\rr^d$ is an \textit{envy-free partition} for $\mu_1,\dots,\mu_n$ if there is a permutation $\pi:[n]\rightarrow[n]$ such that $\mu_i(C_{\pi(i)})\geq \mu_i(C_{\pi(i')})$ for all $i,i'\in [n]$.  In other words, we can assign each piece to a measure such that each measure $\mu$ receives a piece with maximal $\mu$-measure.  There are many extensions of the cake-splitting theorem, involving more relaxed conditions on the preferences of the players \cite{Meunier2019}, secretive guests \cite{Asada:2018ix}, more guests than pieces in the partition \cite{Soberon2022}, more pieces than guests \cite{Nyman2020}, and divisions of undesirable resources \cites{Su:1999es, Avvakumov2023}.  The algorithmic versions are also of substantial interest \cites{Brams:1996wt, procaccia2015cake}.

For equipartitions, the number of measures we can split evenly is typically bounded by the dimension (or in some cases, a function of the dimension). By relaxing the fairness condition, an envy-free partition can satisfy more observers than the dimension. For example, \cref{thm:cake-partition} shows the existence of an envy-free partition of a one-dimensional resource for an \textit{arbitrary number of observers}.  In this manuscript, we will prove high-dimensional versions of this envy-free result.

A simple example of our results is as follows:

\begin{theorem}\label{thm:R2-two-lines}
    Given four absolutely continuous probability measures on $\rr^2$, there exists an envy-free partition that splits $\rr^2$ into four pieces using two lines.  Moreover, the partition can be found even if we only have access to three of the four measures.
\end{theorem}

In other words, once we split $\rr^2$ into four pieces using two lines, we will be able to assign each piece to a measure such that each measure receives a piece with maximal measure.  Since we don't have access to the fourth measure $\mu_4$, this means that regardless of which piece has the largest value for $\mu_4$, there is a permutation assigning that piece to $\mu_4$ which generates an envy-free partition for all four measures.

If all measures are equal, this is a folklore result first noted by Courant and Robbins \cite{Courant1941}.  We prove a wide range of results of this kind, in which the spaces of partitions can be parametrized by a simplex of a sufficiently high dimension.  These include nested hyperplane partitions with fixed directions and generalized voronoi diagrams.  Nested hyperplane partitions with fixed directions have been used before to establish high-dimensional versions of the necklace splitting theorem \cites{Karasev:2016cn, Blagojevic:2018gt}.  In the necklace splitting problem, the goal is to find an equipartition of several measures among $k$ players by splitting the space into potentially more than $k$ pieces and then distributing them among the players \cites{Goldberg:1985jr, Alon:1987ur, Longueville2008}.

We also prove more general results for convex partitions of $\rr^d$.  A convex partition of $\rr^d$ into $n$ parts is a collection $C_1, \dots, C_n$ of closed convex sets in $\rr^d$ such that
\begin{itemize}
    \item the union of the $n$ sets is $\rr^d$, and
    \item the interiors of $C_1, \dots, C_n$ are pairwise disjoint.
\end{itemize}

We prove the existence of partitions which are fair partitions for several groups of measures simultaneously. 

\begin{theorem}\label{thm:envy-free-convex-simple}
    Let $n,d$ be positive integers, where $n$ is a prime power.  Let $\mu$ be an absolutely continuous probability measure on $\rr^d$, and let $(\mu^1_1, \dots, \mu^1_n), \dots, (\mu^{d-1}_1, \dots, \mu^{d-1}_n)$ be $d-1$ tuples of $n$ absolutely continuous probability measures on $\rr^d$ each.

    Then, there exists a convex partition $C_1, \dots, C_n$ of $\rr^d$ such that
    \[
    \mu(C_1) = \dots = \mu(C_n)
    \]
    and $(C_1, \dots, C_n)$ is an envy-free partition for $(\mu^r_1,\dots, \mu^r_n)$, for each $r=1,\dots, d-1$.  Moreover, the partition can be found even if we do not have access to $\mu^{r}_n$ for each $r\in[d-1]$.
\end{theorem}

The case when $\mu^r_j = \mu^r_{j'}$ for each $r \in [d-1]$ and any $j, j' \in [n]$ (each $n$-tuple just has a single measure repeated $n$ times) shows that for any $d$ measures on $\rr^d$, there exists a convex equipartition.  For this equipartition result, the case for general $n$ follows from the case with $n$ being a prime power via a simple subdivision argument.  This is a known extension of the ham sandwich theorem, with several different proofs \cites{Soberon:2012kp, Karasev:2014gi, Blagojevic:2014ey}, originally motivated by the Nandakumar--Ramana-Rao problem \cite{Nandakumar2012}.

The prime power condition in \cref{thm:envy-free-convex-simple} is due to the use of a much stronger theorem of Blagojevi\'c and Ziegler \cite{Blagojevic:2014ey} in our proof, for which this condition is essential.  However, we conjecture that \cref{thm:envy-free-convex-simple} should still hold for general $n$.  We can remove the prime power condition if we are satisfied with a fairness condition that is slightly weaker than envy-freeness.  Given $n$ absolutely continuous probability measures $\mu_1, \dots, \mu_n$ on $\rr^d$, we say that a partition $(C_1,\dots, C_n)$ of $\rr^d$ is a \textit{proportional partition} if there is a permutation $\pi:[n]\to[n]$ such that $\mu_i (C_{\pi(i)}) \ge 1/n$.  Note that envy-freeness implies proportionality, but the converse does not necessarily hold. In a proportional partition, each measure receives a piece with at least the average value, but this is not guaranteed to be the piece with the greatest value.

\begin{theorem}\label{thm:proportional-convex-simple}
    Let $n,d$ be positive integers.  Let $\mu$ be an absolutely continuous probability measure on $\rr^d$, and let $(\mu^1_1, \dots, \mu^1_n), \dots, (\mu^{d-1}_1, \dots, \mu^{d-1}_n)$ be $d-1$ tuples of $n$ absolutely continuous probability measures on $\rr^d$ each.

    Then, there exists a convex partition $C_1, \dots, C_n$ of $\rr^d$ such that
    \[
    \mu(C_1) = \dots = \mu(C_n)
    \]
    and $(C_1, \dots, C_n)$ is a proportional partition for $(\mu^r_1,\dots, \mu^r_n)$ for each $r\in[d-1]$.
\end{theorem}

We also present a generalization of \cref{thm:envy-free-convex-simple} in \cref{sec:remarks}, in which the preferences are not necessarily derived from measures.

\section{Preliminaries}
In this section we discuss two tools that will be needed in our proofs: KKM covers and power diagrams.

Consider the $(n-1)$-dimensional simplex $\Delta^{n-1} \subset \rr^d$ as the set
\[
\Delta^{n-1} = \{(x_1, \dots, x_n) : x_1 + \dots + x_n = 1 \quad \mbox{and } x_i \ge 0 \mbox{ for all }i \in [n]\}.
\]
Denote the vertices of $\Delta^{n-1}$ by the numbers $1,\dots, n$.  For a collection of vertices $I \subset [n]$, let $\sigma_I$ be the face spanned by those vertices.  A \textit{KKM cover} (named after Knaster, Kuratowski, and Mazurkiewicz \cite{Knaster:1929vi}) of $\Delta^{n-1}$ is an $n$-tuple of sets $(A_1, \dots, A_n)$ such that
\begin{itemize}
    \item each $A_i$ is a closed subset of $\Delta^{n-1}$ and
    \item for each $I \subset [n]$, the face $\sigma_I$ is contained in $\bigcup_{i \in I} A_i$.
\end{itemize}
The second condition implies that the sets $A_1, \dots, A_n$ cover $\Delta^{n-1}$, by taking $I=[n]$.  The first condition is many times replaced by all $A_i$ being open set.  KKM covers are commonly used to prove existence results for envy-free partitions and their duals (e.g. \cites{Su:1999es, Meunier2019, Asada:2018ix, Soberon2022} and the references therein).  The main result we will use is the following theorem:

\begin{theorem}[Asada, Frick, Psharody, Polevy, Stoner, Tsang, Wellner 2018 \cite{Asada:2018ix}]\label{thm:asada}
    Let $n$ be a positive integer and $(A^1_1, \dots, A^1_n), \dots, (A^{n-1}_1, \dots, A^{n-1}_n)$ be KKM covers of $\Delta^{n-1}$.  Then, there exists a point $x \in \Delta^{n-1}$ such that for any KKM cover $(A^n_1, \dots, A^n_n)$, there exists a permutation $\pi:[n] \to [n]$ such that
    \[
    x \in A^{i}_{\pi(i)} \mbox{ for each } i \in [n].
    \]
\end{theorem}

This result generalizes the classic ``colorful'' KKM theorem of Gale \cite{Gale1984}, in which $(A^n_1, \dots, A^n_n)$ is also given in advance.

A particular space of partitions of $\rr^d$ we are interested in are \textit{power diagrams}, also known as \textit{generalized Voronoi diagrams}.  Given $n$ distinct points $x_1, \dots, x_n$ in $\rr^d$ (referred to as ``sites'') and $n$ real numbers $\lambda_1, \dots, \lambda_n$ (referred to as ``weights''), this induces a partition of $\rr^d$ into $n$ sets $C_1, \dots, C_n$ defined by

\[
C_j = \{y \in \rr^d : \operatorname{dist}^2(y,x_j) - \lambda_j \le \operatorname{dist}^2(y,x_i) - \lambda_i \mbox{ for all }i\in [n]\}.
\]

These regions are convex.  Since adding the same constant to all weights does not change the partition, we may assume that $\sum_{i=1}^n \lambda_i = 0$ for simplicity.  The case $\lambda_1 = \dots = \lambda_n = 0$ corresponds to a classic Voronoi diagram.  

Given an absolutely continuous probability measure $\mu \in \rr^d$ and $n$ sites $x_1, \dots, x_n$, there exist weights $\lambda_1, \dots, \lambda_n$ such that $\mu(C_i) = 1/n$, as first shown by Aurenhammer, Hoffman, and Aronov \cite{Aurenhammer:1998tj}.  Moreover, if $\mu(U) \neq 0$ for any open set $U$, then this set of weights is unique and varies continuously with $(x_1, \dots, x_n)$.  Denote by $\EMP(\mu, n)$ the space of ``equal measure partitions'' of $\mu$ by power diagrams into $n$ parts. Then, this allows us to parametrize $\EMP(\mu, n)$ by the set of $n$-tuples of $n$ distinct sites in $\rr^d$. The space $\EMP(\mu, n)$ also has a natural action of the symmetric group $S_n$ of permutations on $[n]$.  Given $P=(C_1,\dots, C_n)$ and $\pi \in S_n$, we declare $\pi P = (C_{\pi(1)}, \dots, C_{\pi(n)})$.

A function $F: \EMP(\mu, \rr^n) \to \rr^n$ can be expanded as $F=(f_1,\dots, f_n)$, where each component is a function $f_i:\EMP(\mu, \rr^n) \to \rr$.  We say that $F$ is \textit{$S_n$-equivariant} if $f_i(\pi(P)) = f_{\pi(i)}(P)$ for all $i\in [n]$ and all $\pi\in S_n$.  A simple way to generate such a function is to have each $f_i$ compute some function of $C_i$, such as $\operatorname{\mu}(C_i)$ for some measure $\mu$.  

The main result we use about $S_n$-equivariant functions on $\EMP(\mu,n)$ is the following theorem of Blagojevi\'c and Ziegler.

\begin{theorem}[Blagojevi\'c, Ziegler 2014 \cite{Blagojevic:2014ey}]\label{thm:BZ14}
    Let $\mu$ be an absolutely continuous probability measure on $\rr^d$.  If $n$ is a prime power, then for any $d-1$ continuous $S_n$-equivariant functions $F_r=(f_{1,r},\dots,f_{n,r})$ on $\EMP(\mu, n)$, there exists a partition $(C_1,\dots, C_n)\in \EMP(\mu, n)$ that simultaneously equalizes the components of each $F_r$.  In other words,
\[
f_{i,r}(C_1,\dots, C_n) = f_{i',r}(C_1,\dots, C_n) \qquad \mbox{for all } i,i' \in [n], r \in [d-1].
\]
\end{theorem}

In our application of this result, we will use functions where each $f_{i,r}$ depends on the entire partition $(C_1,\dots,C_n)$ and not just $C_i$.

Given an absolutely continuous probability measure $\mu$ on $\rr^d$, the space $X(\mu,n)$ of convex equipartitions $(C_1,\dots, C_n)$ of $\mu$ is compact.  To see this, consider a ball $B$ such that $\mu(B) > (n-1)/n$.  Then, for any convex equipartition $(C_1, \dots, C_n)$, the hyperplane $H_{ij}$ that separates $C_i$ and $C_j$ (which can be chosen canonically even if $C_i$ and $C_j$ do not share a boundary) must intersect $B$.  Let $K$ be the set of hyperplanes that intersect $B$.  Therefore, we can represent $X$ as a closed subset of
\[
\underbrace{K \times K \times \dots \times K}_{\binom{n}{2} \mbox{ times}}.
\]

Additionally, $X$ inherits a metric space structure from this representation.  This will be useful in \cref{sec:remarks}.


\section{Results parametrized by simplices}

We start by proving \cref{thm:R2-two-lines}.  The parametrization below is based on recent work by McGinis and Zerbib for problems related to line transversals to finite families of convex sets \cite{McGinnis2022} (see also the extensions by G\'omez-Navarro and Rold\'an-Pensado \cite{GomezNavarro2023}).

\begin{proof}[Proof of \cref{thm:R2-two-lines}]
    First, we parametrize the space of partitions of $\rr^2$ using two lines by $\Delta^3$.  Given $(x_1, x_2, x_3, x_4) \in \Delta^3$, consider the following points in the unit circle $S^1$:
    \begin{align*}
    p_0 & = (1,0), \\ p_1 & = \Big(\cos (2\pi x_1), \sin (2\pi x_1)\Big),\\ p_2 & = \Big(\cos (2\pi (x_1+x_2), \sin (2\pi (x_1+x_2)\Big) , \\    p_3 & = \Big(\cos (2\pi (x_1+x_2+x_3), \sin (2\pi (x_1+x_2+x_3)\Big).
    \end{align*}

Note that the length of the arc between $p_{i-1}$ and $p_i$ (indices modulo $4$) is exactly $2\pi x_i$.  Let $B^2$ be the disk bounded by $S^1$.  We construct the lines $p_0p_2$ and $p_1p_3$, which divide the disk into four closed regions.  Denote by $C_i$ the region of the disk that contains the arc $p_{i-1}p_i$.  We first prove the result for measures whose support is contained in $B^2$.

Given an absolutely continuous probability measure $\mu$ whose support is contained in $B^2$, we construct a cover of $\Delta^3$ as following.  Given $x=(x_1, x_2, x_3, x_4) \in \Delta^3$, construct $C_1, C_2, C_3, C_4$ as above.  We say $x \in A_i$ if $\mu(C_i) \ge \mu (C_j)$ for all $j \in [4]$.  Since the condition involves checking a finite number of non-strict inequalities, each $A_i$ is a closed set.  Moreover, if $x_i=0$, the set $C_i$ is the single point $p_i$, so $\mu(C_i) = 0$.  This implies that $(A_1, A_2, A_3, A_4)$ is a KKM cover of $\Delta^3$.  Then the four measures induce four corresponding KKM covers of $\Delta^3$.  A direct application of \cref{thm:asada} finishes the proof.

Now consider the case of general measures and let $\varepsilon>0$.  We may assume, without loss of generality, that for each measure $\mu_1, \dots, \mu_4$ we have $\mu_i (B^2) \ge 1-\varepsilon$ by, say, scaling the plane or the disk.  We can apply the result above and find two lines generating an envy-free partition for the restriction of the measures to $B^2$.  We call this a partition that is $\varepsilon$-fair, as the piece assigned to each $\mu_i$ needs at most $\varepsilon$ measure to be the maximal part.  Observe that the space of pairs of lines that could represent a $\delta$-fair partition for the four measures for any $\delta \le 1/2$ is compact.  A simple compactness argument by taking $\varepsilon \to 0$ gives us an envy-free partition as desired.

\begin{figure}
    \centering
    \includegraphics[scale=0.7]{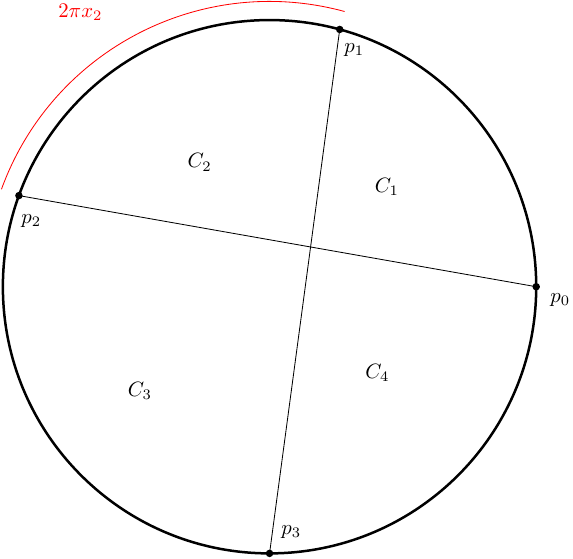}
    \caption{construction of the two lines using $(x_1, x_2, x_3, x_4)$}
    \label{fig:twolines}
\end{figure}
    
\end{proof}

In general, the technique above works for a wide range of partition spaces parametrized by a simplex, which we denote as \textit{$\Delta$-spaces}.

\begin{definition}
    Let $n, d$ be positive integers.  We say that a space of of partitions $\mathcal{H}$ of $\rr^d$ into $n$ parts is a $\Delta$-space if there exists a function $R: \Delta^{n-1}\to \mathcal{H}$ such that the following holds:
    \begin{itemize}
        \item If $R(x_1,\dots,x_n) = (C_1, \dots, C_n)$, then $x_i = 0$ implies that $C_i$ has Lebesgue measure zero for all $i \in [n]$, and
        \item if $\mu$ is a finite measure absolutely continuous with respect to the Lebesgue measure, then $(x_1, \dots, x_n) \mapsto \mu(C_j)$ is a continuous function from $\Delta^{n-1}$ to $\rr$ for each $j \in [n]$.
    \end{itemize}
\end{definition}

Now, given two $\Delta$-spaces of partitions of $\rr^d$, there is a simple way to combine them.  Let $\mathcal{A}$ be a $\Delta$-space of partitions of $\rr^d$ into $n$ pieces, let $\mathcal{B}$ be a $\Delta$-space of partitions of $\rr^d$ into $m$ pieces, and let $v$ be a direction in $S^{d-1}$.  Using these parameters, we can construct a $\Delta$-space of partitions $\mathcal{A} *_v \mathcal{B}$ of $\rr^d$ into $n+m$ pieces.

Given $\beta \in \rr$ and $v \in [0,1]$, consider the half-spaces
\begin{align*}
    H^+(\beta,v) = \{y \in \rr^d: \langle y, v \rangle \ge \beta \} \\
    H^-(\beta,v) = \{y \in \rr^d: \langle y, v \rangle \le \beta \}
\end{align*}

We include the limiting cases $H^+(\infty, v) = H^- (-\infty, v) = \emptyset$ and $H^-(\infty, v) = H^+(-\infty, v) = \rr^d$.

Now, given partitions $(C_1,\dots, C_n) \in \mathcal{A}$, $(D_1, \dots, D_m) \in \mathcal{B}$ and $\alpha \in [0,1]$, we form a new partition with $n+m$ sets, of the form
\[
C'_j = C_j \cap H^+ \left(\frac{2\alpha-1}{1-|2\alpha-1|}, v\right), j \in [n] \qquad D'_k=D_k \cap H^- \left(\frac{2\alpha-1}{1-|2\alpha-1|},v\right), k \in [m].
\]

The construction is illustrated in \cref{fig:join-partition}.  Some parts may be empty, which is not a problem.  Note that as $\alpha \to 0$, we recover only the partition in $\mathcal{A}$ and $m$ empty pieces for $\mathcal{B}$, and as $\alpha \to 1$ we recover only the partition of $\mathcal{B}$ and $n$ empty pieces for $\mathcal{A}$.  This means that if $\mathcal{A}$ and $\mathcal{B}$ have parametrizations of $\Delta^{n-1}$ and $\Delta^{m-1}$, respectively, then $\mathcal{A} *_v \mathcal{B}$ can be parametrized by the join $\Delta^{n-1}*\Delta^{m-1} = \Delta^{n+m-1}$.  Hence $\mathcal{A}*_v \mathcal{B}$ is also a $\Delta$-space, as pieces that must have Lebesgue measure zero in the construction preserve this property.  This method of combining spaces of partitions has been used before to obtain high-dimensional generalizations of the necklace splitting theorem.

\begin{figure}
    \centering
    \includegraphics{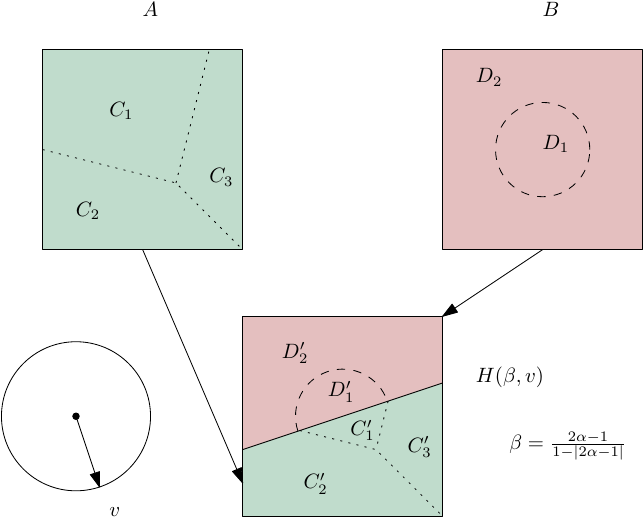}
    \caption{A construction of a partition of $\rr^2$ into five parts.}
    \label{fig:join-partition}
\end{figure}

A simple example of a $\Delta$-space is to simply take $n=1$ and the unique partition $\{\rr^d\}$.  Then, using the join operation repeatedly with potentially different directions $v$, we obtain the space of \textit{nested hyperplane partitions} with fixed directions.  Alternatively, we can define the same nested hyperplane partition space of $\rr^d$ into $n$ parts recursively.  For $n=1$, it is simply the partition $\{\rr^d\}$.  For $n>1$, assume we have a space $\mathcal{A}$ of nested hyperplane partitions into $n-1$ parts, a direction $v \in S^{d-1}$, and a value $j \in [n-1]$.  Given a partition $(C_1, \dots, C_{n-1})\in \mathcal{A}$, we simply cut $C_j$ into two parts $C_j'$ and $C_j''$ using a hyperplane orthogonal to $v$ (we always cut the $j$-th piece).  This gives us a partition of $\rr^d$ into $n$ pieces, as desired.  Note that the hyperplanes involved in the cutting have fixed directions, and by recursion the parts are convex. 

For any space of nested hyperplane partitions into $n$ pieces for $n\ge 2$, there exists $1\le k \le n-1$ such that one side of the first hyperplane cut is a nested hyperplane partition into $k$ parts and the other side is a nested hyperplane partition into $n-k$ parts.  Therefore, by strong induction we can parametrize this space of partitions by $\Delta^{k-1}*\Delta^{n-k-1} \cong \Delta^{n-1}$.

This gives us the following consequence of \cref{thm:asada}.  Note that it implies a generalization of \cref{thm:R2-two-lines}.

\begin{theorem}\label{thm:delta-spaces}
    Let $\mathcal{B}$ be a $\Delta$-space of partitions of $\rr^d$ into $n$ parts.  Let $\mu_1, \dots, \mu_n$ be absolutely continuous probability measures on $\rr^d$.  Then, there is a partition in $\mathcal{B}$ that is an envy-free partition for $\mu_1,\dots, \mu_n$.  The partition can be found even if we don't have access to $\mu_n$ in advance.
\end{theorem}

\begin{proof}
    For each $\mu_j$, we define a KKM cover $(A^j_1,\dots, A^j_n)$ of $\Delta^{n-1}$.  First, we use the function $R:\Delta^{n-1}\to \mathcal{B}$ from the definition of a $\Delta$-space.  Then, we define $A^j_i$ to be the set of points $(x_1,\dots, x_n) \in \Delta^{n-1}$ such that for $(C_1,\dots,C_n) = R(x_1,\dots,x_n)$, we have that $C_i$ has the maximal $\mu_j$-measure among all $C_i$.  The condition of Lebesgue measure zero implies that $A^j_i$ does not intersect the facet opposite to vertex $i$.  Since each $A^j_i$ is closed, this is a KKM cover of $\Delta^{n-1}$.  A direct application of \cref{thm:asada} finishes the proof.
\end{proof}

If we apply the result for nested hyperplane partitions, we get the following corollary.

\begin{corollary}
    Let $d, n \ge 1$ and let $A$ be a space of nested hyperplane partitions of $\rr^d$ with fixed directions.  Then, for any $n$ absolutely continuous measures, there exists an envy-free partition in $A$ for the $n$ measures.  This partition can be found even if we don't have access to one of the measures in advance.
\end{corollary}

Although power diagrams with fixed sites are not $\Delta$-spaces since they do not meet the Lebesgue measure zero condition, they satisfy a similar condition for finite measures.  Let $(y_1,\dots,y_n)$ be $n$ distinct points in $\rr^d$.  For an $n$-tuple $(\alpha_1,\dots, \alpha_n)$ of real numbers with sum equal to zero, consider $(C_1,\dots, C_n)$ the power diagram with sites $(y_1,\dots, y_n)$ and weights $(\alpha_1,\dots, \alpha_n)$ with zero sum.  Given any absolutely continuous probability measure $\mu$ on $\rr^d$ and $\varepsilon>0$, there exists a real number $M=M(\mu,\varepsilon)$ such that $\alpha_i < M$ implies $\mu(C_i)<\varepsilon$ for all $i\in[n]$.  The set
\[
\left\{ (\alpha_1,\dots, \alpha_n) : \sum_{i=1}^n\alpha_i = 0, \quad \alpha_i \ge M \mbox{ for all\;} i \in [n]\right\}
\]
is a rescaling of $\Delta^{n-1}$.  By choosing $\varepsilon$ appropriately (in this case, any $0 < \varepsilon < 1/n$), we can include power diagrams in \cref{thm:delta-spaces}.  In other words, we have the following corollary.

\begin{corollary}
Let $n, d$ be positive integers.  Given $n$ absolutely continuous probability measures in $\rr^d$ and a fixed set of $n$ sites, there exist weights such that the corresponding power diagram gives an envy-free partition for the $n$ measures.
\end{corollary}

This holds because the induced cover of $\delta^{n-1}$ will be a KKM cover, as the set $A_i$ will not intersect the facet opposite to vertex $i$.

We also prove a generalization of Levi's cone partition result.  Levi's theorem is the first high-dimensional mass partition result, predating the ham sandwich theorem \cite{Levi:1930ea}.  

\begin{theorem}[Levi 1930]
    Let $(C_1, \dots, C_{d+1})$ be a convex partition of $\rr^d$ such that each $C_i$ is a convex cone, and each cone has its apex at the origin.  Let $\mu$ be an absolutely continuous probability measure on $\rr^d$.  Then, there exists a vector $x \in \rr^d$ such that the translates $x + C_i$ all satisfy
    \[
\mu (x + C_i) = \frac{1}{n}.
    \]
\end{theorem}

This result has been generalized by Borsuk and by Vre\'cica and \v{Z}ivaljevi\'c \cites{borsuk1953application, Vrecica:2001dh}.  We prove an envy-free version of their generalizations.

\begin{theorem}\label{thm:strong-levi}
    Let $(C_1, \dots, C_{d+1})$ be a convex partition of $\rr^d$ such that each $C_i$ is a convex cone and all cones have their apex at the origin, and let $\mu_1, \dots, \mu_{d+1}$ be absolutely continuous measures on $\rr^d$.  Let $\alpha_1, \dots, \alpha_{d+1}$ be positive real numbers that sum to $1$.  Then, there exists a vector $x$ and a permutation $\pi:[d+1] \to [d+1]$ such that each translate $x+C_i$ satisfies
    \[
\mu_{\pi(i)}(x + C_i) \ge \alpha_i.
    \]
    Moreover, the vector $x$ can be found even if we don't have access to $\mu_{d+1}$ in advance.
\end{theorem}

If $\mu_1 = \dots = \mu_{d+1}$, the result above is the generalization by Borsuk and by \v{Z}ivaljevi\'c and Vre\'cica.  

\begin{proof}[Proof of \cref{thm:strong-levi}]
    For each $C_i$, we can find a half-space $H^-_i$ such that $C_i \subset H_i$ and its boundary hyperplane contains the origin.  Since the union of all $C_i$ covers $\rr^d$, so does the union of the half-spaces $H^-_i$.  Let $G^-_i$ be a translate of $H^-_i$ so that $\mu_j (G^-_i) < \alpha_i$ for all $j$.  Let $G^+_i$ be the closed half-space opposite to $G^-_i$.  Now consider the simplex $\bigcap_{i=1}^{d+1} G^+_i$.  \cref{fig:levi} illustrates the construction in the plane.  We denote by $F_i$ the facet contained in $G^-_i$, and $v_i$ the vertex opposite to $F_i$.  Now, each measure $\mu_j$ induces a KKM cover $(A^j_1, \dots, A^j_{d+1})$ of $\Delta$ in the following way.  We have $x \in A^j_i$ if and only if $\mu_j(x + C_i) \ge \alpha_i$.  Note that for $x \in F_i$, by construction we have $x \not\in A^j_i$.  Also, since $\sum_{i=1}^{d+1} \mu_j (x + C_i) = 1 = \sum_{i=1}^{d+1} \alpha_i$, then for at least one $i$ we must have $\mu_j (x + C_i) \ge \alpha_i$, so $(A^j_1, \dots, A^j_{d+1})$ indeed covers $\Delta$.  Now we simply apply \cref{thm:asada} to the $d+1$ covers to obtain the desired result.
\end{proof}

\begin{figure}
    \centering
    \includegraphics{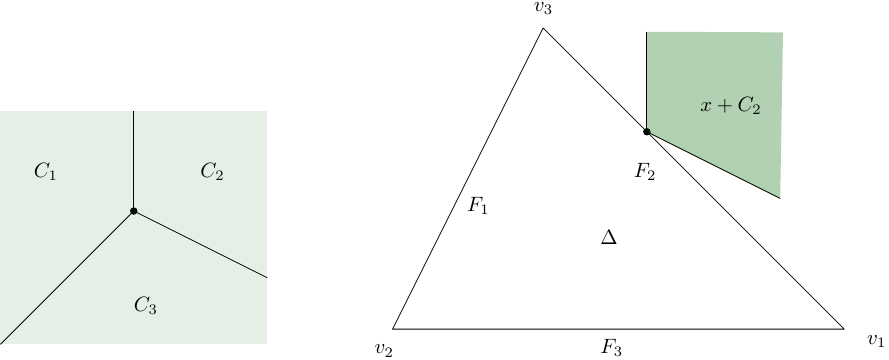}
    \caption{A construction of $\Delta$ in $\rr^2$}
    \label{fig:levi}
\end{figure}





\section{Convex partitions in high dimensions}

We prove the existence of simultaneous proportional and envy-free partitions for several groups of measures on $\rr^d$. First, if the number of convex pieces $n$ is a prime power, we prove the existence of simultaneous envy-free partitions of $\rr^d$.  This is the stronger fairness constraint in our results.





\begin{proof}[Proof of \cref{thm:envy-free-convex-simple}]
Let us briefly explain the structure of the proof.  Given a partition $P=(C_1,\dots, C_n) \in \EMP(\mu, n)$, for each $n$-tuple $(\mu^r_1, \dots, \mu^r_{n})$ of measures we will construct an $n\times n$ matrix $M^r(P)$.  \cref{thm:BZ14} will guarantee that for some $P \in \EMP(\mu,n)$, each of the $d-1$ matrices we constructed will be doubly stochastic, which we will use to finish the proof.  We have to be careful that our construction of $P$ does not depend on $\mu^r_n$ for $r \in [d-1]$ to fulfill the requirements of \cref{thm:envy-free-convex-simple}.

Let $\varepsilon >0$ be a real number.    Given a partition $P=(C_1,\dots,C_n)\in \EMP(\mu, n)$, and $i \in [n], j \in [n-1]$
    \begin{equation*}
        g_{ij}^r(P)=\begin{cases}
            0 &\text{if\;} \mu_j^r(C_i)<\max\{\mu_j^r(C_{i'}) \text{\;for\;} i'\in [n]\}-\epsilon\\
            \mu_j^r(C_i)-(\max\{\mu_j^r(C_{i'}) \text{\;for\;} i'\in [n]\}-\epsilon) &\text{otherwise}
        \end{cases}
    \end{equation*}
    and for $j = n$ we define
    \[
g^r_{in}(P) = \frac{1}{n}.
    \]
    Note that $g^r_{ij}(P)$ is non-negative and varies continuously with $P$. We have that $g_{ij}^r(P)>0$ if and only if $\mu_j^r(C_i)>\mu_j^r(C_{i'})-\epsilon$ for all $i'\in [n]$.  If we fix $j$, there exists at least one value $i \in [n]$ such that $g^r_{ij}(P) \ge \varepsilon > 0$.  We can consider the values $g^r_{ij}$ as the entries of an $n \times n$ matrix and normalize the columns so that they sum to $1$.  In explicit terms, we define
    \[
    M_{ij}^r(P) = \frac{g_{ij}^r(P)}{\sum_{k=1}^n g_{kj}^r(P)}.
    \]
    The value $M_{ij}^r$ is the entry at the $i$-th row and $j$-th column in the $n \times n$ matrix $M^r(P)$.  Note that $M^r(P)$ changes continuously as $P$ changes in $\EMP(\mu)$.  This gives us $d-1$ matrices $M^1(P), \dots, M^{d-1}(P)$.
    
    Let $$f_i^r(P)=\sum_{j=1}^n M_{ij}^r(P)$$ be the sum of the $i$-th row of $M^r(P)$. By construction, the function $F^r=(f_1^r,\dots,f_n^r):EMP(\mu,n) \to \rr^n$ is a continuous $S_n$-equivariant function.

    By \cref{thm:BZ14}, there exists a partition $P\in\EMP(\mu, n)$ that simultaneously equalizes the components of each of $F^1,\dots,F^{d-1}$:
    $$f_1^r(P)=\dots=f_n^r(P) \text{\; for \:} r=1,\dots,d-1.$$

    Then for this partition, $M^1(P),\dots,M^{d-1}(P)$ are $d-1$ doubly stochastic matrices. By Birkhoff's theorem, each $M^r(P)$ is a convex combination of permutation matrices.  Recall we still have not used $\mu^r_n$.  For any $i' \in [n]$, since $M^r_{i'n} (P) = 1/n > 0$ and $M^r(P)$ is a convex combination of permutation matrices, there exists a permutation $\pi^{i'}_r:[n]\rightarrow [n]$ such that $M_{\pi^{i'}_r(j)j}^r(P)>0$ for all $j\in[n]$ and $\pi^{i'}_r(n) = i'$. This implies that for $j\in[n]$, we have $\mu_j^r(C_{\pi^{i'}_r(j)})>\mu_j^r(C_i)-\epsilon$ for all $i\in [n]$.  The space of convex equipartitions of $\mu$ into $n$ parts (not necessarily from a power diagram) is compact, and the number of possible permutations $\pi^{i'}_r$ is finite.  Therefore, if we take a sequence $\varepsilon_m$ such that $\varepsilon_m \to 0$, we can assume without loss of generality that for each $\varepsilon_m$ each of the permutations $\pi^{i'}_r$ is the same, and that as $\varepsilon_m \to 0$ the partition $P$ approaches a convex equipartition $\overline{P}=(C_1,\dots,C_n)$ of $\mu$.  This equipartition satisfies
    \[
    \mu_j^r(C_{\pi^{i'}_r(j)})\geq\mu_j^r(C_i) \text{\; for all\;} i\in[n]
    \]
    for all $r\in[d-1]$ and $i\in n$.  Finally, for each $r \in [d-1]$, we check which part $C_{i_r}$ has the largest $\mu^r_n$ value.  The permutations $\pi^{i_r}_r$ for $r \in [d-1]$ give us the desired allocation.
\end{proof}

For general $n$, we use a subdivision argument to show the existence of simultaneous proportional partitions of $\rr^d$.  Subdivision arguments are standard in mass partition results.  In this case, the subdivision argument is more nuanced than usual, as it requires a slighly stronger version of \cref{thm:envy-free-convex-simple}.  Given $ms$ participants, each with their own measure, we seek to divide $\rr^d$ into $m$ sets so that each part is preferred by exactly $s$ participants, without overlaps.  Moreover, we want to do this simultaneously for $d-1$ groups of $ms$ participants each.




\begin{lemma}\label{lem:coarsePrP}
    Let $m$ be a prime power, $n$ be a multiple of $m$ and $\mu$ be an absolutely continuous probability measure on $\rr^d$. For any $d-1$ $n$-tuples of absolutely continuous probability measures $(\mu_1^1,\dots,\mu_n^1),\dots,(\mu_1^{d-1},\dots,\mu_n^{d-1})$ on $\rr^d$, there is a partition of $\rr^d$ into $m$ convex regions $C_1,\dots,C_m$ and $d-1$ functions $\pi^r:[n] \to [m]$ for $r \in [d-1]$ such that
    \begin{align*}
    \mu(C_1)=\dots=\mu(C_m)=\frac{1}{m}&  \\
    \mu^r_j (C_{\pi^r(j)}) \ge \frac{1}{m}&  \qquad  \mbox{for all } i \in [m]
    \end{align*}
    and such that for each $i \in [m], r \in [d-1]$, the preimage $(\pi^r)^{-1}(i)$ has exactly $n/m$ elements.
\end{lemma}

The proof below can be modified to have the stronger conclusion $\mu^r_j (C_{\pi^r(j)}) \ge \mu^r_j(C_i)$ for all $i \in [n]$, but it does not give us any advantage for the subdivision argument.  This gives a slightly stronger version of \cref{thm:envy-free-convex-simple} if we have access to all measures in advance.  With the condition $\mu^r_j (C_{\pi^r(j)}) \ge 1/m$, we have a slightly simpler matrix construction.

\begin{proof}
    Let $\varepsilon>0$ be a real number.  Given a partition $P=(C_1,\dots, C_m)\in \EMP(\mu, m)$, define 
    \begin{equation*}
        g_{ij}^r(P)=\begin{cases}
            0 &\text{if\;} \mu_j^r(C_i)<\frac{1}{m}-\epsilon\\
            \mu_j^r(C_i)-(\frac{1}{m}-\epsilon) &\text{otherwise}.
        \end{cases}
    \end{equation*}
    Each of these values is non-negative, and changes continuously as $P$ changes.  We have that $g_{ij}^r(P)>0$ if and only if $\mu_i^r(C_j)>\frac{1}{m}-\epsilon$.

    Let 
    \[
    M_{ij}^r(P) = \frac{m \cdot g_{ij}^r(P)}{n \cdot \sum_{k=1}^m g_{kj}^r(P)}
    \]
    which is well-defined since $g_{kj}^r(P)>0$ if and only if $\mu_j^r(C_k)> 1/m-\epsilon$, which holds for at least one $k\in[m]$. These values are continuous functions of $P$ and $M_{ij}^r(P)>0$ if and only if $\mu_j^r(C_i)>\frac{1}{m}-\epsilon$.
    
    Construct a $m\times n$ matrix $M^r(P)$ such that the entry in the $i$-th row and and $j$-th column is $M_{ij}^r(P)$. The entries of $M^r(P)$ are non-negative and each of its columns sum to $m/n$. 
    
    Let $$f_i^r(P)=\sum_{j=1}^n M_{ij}^r(P)$$ be the sum of the $i$-th row of $M^r(P)$. Note that $F^r=(f_1^r,\dots,f_m^r): \EMP(\mu,m) \to \rr^m$ is an $S_m$-equivariant continuous function.

    By \cref{thm:BZ14}, there exists a partition $P\in\EMP(\mu, m)$ that simultaneously equalizes the components of each of $F^1,\dots,F^{d-1}$:

    Now consider the matrix $M^r(P)$ for any $r \in [d-1]$.  Its column sum is constant and equal to $m/n$, and its row sum is constant.  Therefore, the sum of each row is $1$.   We can construct an $n \times n$ matrix $N^r(P)$ by placing $n/m$ copies of $M^r(P)$ on top of each other.  This gives us an $n \times n$ matrix with non-negative entries whose column sum and row sum are all $1$.  Applying Birkhoff's theorem again, there is a permutation $\sigma:[n]\to [n]$.  Such that $N^r_{\sigma(j)j} > 0$.  We now define $\pi^r(j)$ the unique integer in $[m]$ such that $\pi^r(j)$ is congruent to $\sigma^r(j)$ modulo $m$.  We conclude by a compactness argument taking $\varepsilon \to 0$ as in the proof of \cref{thm:envy-free-convex-simple}.
\end{proof}

With \cref{lem:coarsePrP}, we now prove \cref{thm:proportional-convex-simple}.


\begin{proof}[Proof of \cref{thm:proportional-convex-simple}]
    We proceed via strong induction on $n$. The base case $n=1$ holds by taking $C_1=\rr^d$. Suppose $n\geq 2$ and the theorem holds for all $1\leq n'< n$.
    
    First, if $n$ is a prime power, the theorem holds by Theorem \ref{thm:envy-free-convex-simple}.
    
    Otherwise if $n$ is not a prime power, it can be factored as $n=ms$ with $m,s<n$ and $m$ a prime power. By Lemma \ref{lem:coarsePrP}, there is a partition of $\rr^d$ into $m$ convex regions $C_1,\dots,C_m$ and $d-1$ partitions of $[n]$ into $m$ sets $S_1^r,\dots,S_m^r$ of $s$ indices each such that for $r=1,\dots, d-1$ and $j=1,\dots,m$, we have $\mu_i^r(C_j)\geq \frac{1}{m}=\frac{s}{n}$ for $i\in S_j^r$. Then for each $C_j$, we apply the theorem to the $s$ measures $\mu_i|_{C_j}$ for $i\in S_j^r$, and obtain a partition of $\rr^d$ into $s$ regions $C_{j,1},\dots,C_{j,s}$ and $d-1$ permutations $\pi_{j,1},\dots,\pi_{j,d-1}:[s]\rightarrow[s]$ such that
        $\mu_i^r(C_{j,\pi_r(i)})\geq\frac{1}{n}$. Combining these completes the proof. 
        
        Concretely, let $D_{s(j-1)+h}=C_j\cap C_{j,h}$, which is convex. Define 
        \begin{align*}
            \tilde{\pi}_r:[n]&\rightarrow[n]\\
            i &\mapsto \pi_{j,r}(i-s(j-1)) \text{\; if \;} s(j-1)<i\leq sj \text{\; for \;} j=1,\dots,m.
        \end{align*}
        Then $\mu_i^r(D_{\tilde{\pi}_r(i)})\geq\frac{1}{n}$, so $D_1,\dots,D_n$ and $\tilde{\pi}_1,\dots,\tilde{\pi}_{d-1}$ are the convex partition and permutations, respectively, that we need.
    
\end{proof}

Note that using a subdivision argument weakens our result to be proportional rather than envy-free. An envy-free result would require coordination across all subdivisions at once, to ensure that the piece that one person receives from a given subdivision contains more measure than any other piece produced from the other subdivisions.

\section{Remarks and open problems}\label{sec:remarks}

We conjecture that \cref{thm:envy-free-convex-simple} should hold for general $n$.  Moreover, we also believe that the measure $\mu$ could be replaced by another set of $n$ measures.  Concretely, we have the following conjecture:

\begin{conjecture}\label{conj:envy-free-strong}
    Let $n,d$ be positive integers and let $(\mu^1_1,\dots, \mu^1_n), \dots, (\mu^d_1, \dots, \mu^d_n)$ be $d$ tuples of $n$ absolutely continuous probability measures each.  Then, there exists a convex partition $(C_1, \dots, C_n)$ of $\rr^d$ that is an envy-free partition for $(\mu^r_1, \dots, \mu^r_n)$ for each $j=1,\dots, d$.
\end{conjecture}

In our proof, we used very little information about the measures.  We can prove a result for preferences that do not come from measures in the following way.  Given an absolutely continuous probability measure $\mu$ on $\rr^d$, let $X(\mu, n)$ be the space of convex equipartitions of $\mu$ into $n$ pieces.  A KKM-like cover $\mathcal{U}$ of $X(\mu,n)$ will be an $n$-tuple $(A_1, \dots, A_n)$ of sets of $X(\mu,n)$ such that
\begin{itemize}
    \item the set $A_i \subset X(\mu,n)$ is open for each $i \in [n]$,
    \item the sets $A_1,\dots, A_n$ cover $X(\mu,n)$, and
    \item if $(C_1,\dots, C_n) \in A_i$ and $\pi:[n] \to [n]$ is a permutation, then ${(C_{\pi(1)}, \dots, C_{\pi(n)}) \in A_{\pi^{-1}(i)}}$.  In other words, shuffling the partition does not change the preferences.
\end{itemize}

Given a set $(A^1_1,\dots, A^1_n), \dots, (A^n_1, \dots, A^n_n)$ of $n$ KKM-like covers of $X(\mu,d)$, we say that a partition $(C_1,\dots,C_n)$ is an envy-free partition for the covers if there exists a permutation $\pi:[n] \to [n]$ such that $(C_1,\dots, C_n) \in A^i_{\pi(i)}$ for each $i \in [n]$.  Then, we have the following theorem:

\begin{theorem}\label{thm:envy-free-strong}
    Let $n,d$ be positive integers, where $n$ is a prime power.  Let $\mu$ be an absolutely continuous probability measure on $\rr^d$, and let $(\mathcal{U}^1_1, \dots, \mathcal{U}^1_n), \dots, (\mathcal{U}^{d-1}_1, \dots, \mathcal{U}^{d-1}_n)$ be $d-1$ KKM-like covers of $X(\mu,n)$.

    Then, there exists a convex partition $C_1, \dots, C_n$ of $\rr^d$ such that
    \[
    \mu(C_1) = \dots = \mu(C_n)
    \]
    and $(C_1, \dots, C_n)$ is an envy-free partition for $(\mathcal{U}^r_1, \dots, \mathcal{U}^r_n)$ for each $r=1,\dots, d-1$.  Moreover, the partition can be found even if we do not have access to $\mathcal{U}^r_n$ for each $r\in[d-1]$.
\end{theorem}

\begin{proof}
    The proof is identical to the proof of \cref{thm:envy-free-convex-simple} except for the definition of $g^r_{ij}(P)$ and without the need for the compactness argument at the end (so $\varepsilon$ is not used).  For a fixed $r,j$, let  Let $\mathcal{U}^r_j = (A_1,\dots, A_n)$.  Recall that $A_i \subset X(\mu,d)$ is an open set, and that $X(\mu,d)$ is a metric space.  Then, for $j \in [n-1]$
    \[
    g^r_{ij}(P) = \operatorname{dist}(A_i^c, P), 
    \]
    where $A^c_i$ denotes the complement of $A_i$.  Once we have constructed $g^r_{ij}$, the rest of the proof follows.
\end{proof}


\end{document}